\documentclass{amsart}
\pdfoutput=1

\usepackage{amssymb} 
\usepackage{amsmath} 
\usepackage{amscd}
\usepackage{amsbsy}
\usepackage{comment}
\usepackage[matrix,arrow]{xy}
\usepackage{url}
\usepackage{hyperref}

\newtheorem{theorem}{Theorem}[section]
\newtheorem{lemma}[theorem]{Lemma}
\newtheorem{corollary}[theorem]{Corollary}

\newtheorem{dfn}{Definition}
\usepackage{epsfig}

\usepackage{enumerate}
\def\mod#1{{\ifmmode\text{\rm\ (mod~$#1$)}
\else\discretionary{}{}{\hbox{ }}\rm(mod~$#1$)\fi}}

\begin{document}
\title{Squares with three nonzero digits}

\author{Michael A. Bennett}
\address{Department of Mathematics, University of British Columbia, Vancouver BC}
\email{bennett@math.ubc.edu}
\thanks{The authors were supported in part by  grants from NSERC}

\author{Adrian-Maria Scheerer}
\address{Department for Analysis and Computational Number Theory, Graz University of Technology, Graz, Austria}
\email{scheerer@math.tugraz.at}
\thanks{The second author was supported by the Austrian Science Fund (FWF): I 1751-N26; W1230, Doctoral Program ``Discrete Mathematics''; and  SFB F 5510-N26}


\date{\today}
\begin{abstract}
We determine all integers $n$ such that $n^2$ has at most three base-$q$ digits for $q \in \{ 2, 3, 4, 5, 8, 16 \}$. More generally, we show that all solutions to  equations of the shape
$$
Y^2 = t^2 + M \cdot q^m + N \cdot q^n,
$$
where $q$ is an odd prime, $n > m > 0$ and $t^2, |M|, N < q$, either arise from ``obvious'' polynomial families or satisfy $m \leq 3$.
Our arguments rely upon Pad\'e approximants to the binomial function, considered $q$-adically.
\end{abstract}

\maketitle

\section{Introduction}

Let us suppose that $q > 1$ is an integer. A common way to measure the lacunarity of the base-$q$ expansion of a positive integer $n$ is through the study of functions we will denote by $N_q(n)$ and $S_q(n)$, the number of and sum of the nonzero digits in the  base-$q$ expansion of $n$, respectively. Our rough expectation is that, if we restrict $n$ to lie in a subset  $S \subset \mathbb{N}$, these quantities should behave in essentially the same way as for unrestricted integers, at least provided the subset is not too ``thin''. Actually quantifying such a statement can be remarkably difficult; particularly striking successes along these lines, for $S$ the sets of primes and squares can be found in work of Mauduit and Rivat \cite{MR1} and \cite{MR2}.

In this paper, we will restrict our attention to the case where $S$ is the set of integer squares. Since (see \cite{De})
$$
\sum_{n < N} S_q(n) \sim \frac{1}{2} \sum_{n < N} S_q(n^2) \sim \frac{q-1}{2 \log q} N \log N,
$$
it follows that the ratios 
$$
\frac{S_q(n^2)}{S_q(n)} \; \; \mbox{ and } \; \; \frac{N_q(n^2)}{N_q(n)}
$$
are infrequently ``small''. On the other hand, in the case $q=2$ (where $S_q(n)$ and $N_q(n)$ coincide), Stolarsky \cite{Sto} proved that, for infinitely many $n$,
$$
\frac{N_2(n^2)}{N_2(n)} \leq \frac{4 \left( \log \log n \right)^2}{\log n},
$$
a result that was subsequently substantially sharpened and generalized by Hare, Laishram and Stoll \cite{HLS1}. Further developments are well described in \cite{HLS2}  where, in particular, one finds that
$$
\# \left\{ n < N  \; : \; N_2(n) = N_2(n^2) \right\} \gg N^{1/19}
$$
and that the set
$$
\left\{ n \in \mathbb{N}, \, n \mbox{ odd }  \; : \; N_2(n) = N_2(n^2)=k \right\}
$$
is  finite for $k \leq 8$ and  infinite for $k \in \{ 12, 13 \}$ or $k \geq 16$.

In what follows, we will focus our attention on  integers $n$ with the property that $N_q(n^2)=k$, for small fixed positive integer $k$. Classifying those integers $n$ in the set
$$
B_k (q) = \left\{ n \in \mathbb{N} \; : \; n \not\equiv 0 \mod{q} \; \mbox{ and } \; N_q(n) \geq N_q(n^2)=k \right\}
$$
is, apparently, a rather hard problem, even for the case $k=3$ (on some level,  this is the smallest ``nontrivial'' situation as those $n$ with $N_q(n^2) < 3$ are readily understood). There are infinitely many squares, coprime to $q$ with precisely three nonzero digits base-$q$, as evidenced by the identity
\begin{equation} \label{lumpy}
\left( 1 + q^b \right)^2 = 1 + 2 \cdot q^b + q^{2b}.
\end{equation}
There are, however, other squares with three nonzero digits, arising more subtly. For example, if $n=10837$, then, base $q=8$, we have
$$
10837 = 2 \cdot 8^4 + 5 \cdot 8^3 + 1 \cdot 8^2 + 2 \cdot 8 +5
$$
while
$$
10837^2 = 7 \cdot 8^8 + 7 \cdot 8 +1.
$$
On the other hand, a result of Corvaja and Zannier \cite{CoZa00} implies that
all but finitely many squares with three base-$q$ digits arise 
from polynomial identities like (\ref{lumpy}), and, further, that  $B_3(q)$ is actually finite. The proof of this in \cite{CoZa00}, however, depends upon Schmidt's Subspace Theorem and is thus ineffective (in that it does not allow one to precisely determine $B_3(q)$ -- it does, however, lead to an algorithmic determination of all relevant polynomial identities, if any). Analogous questions for $B_k(q)$ with $k \geq 4$ are, as far as we are aware, unsettled, except for the case of $B_4(2)$ (see \cite{CZ11}).

In this paper, we will explicitly determine $B_3 (q)$ for certain fixed values of $q$. We prove the following theorem.
\begin{theorem} \label{Cor1}
The only positive integers $n$ for which $n^2$ has at most three nonzero digits base $q$ for $q \in \{ 2, 3, 4, 5, 8, 16 \}$  and $n \not\equiv 0 \mod{q}$ are as follows :
$$
q=2 \; : \; n \in \{ 1, 5, 7, 23 \} \mbox{ or } n= 2^b+1, 
$$
$$
q=3 \; : \; n \in \{ 1, 5, 8, 13 \} \mbox{ or } n=  3^b+1, 
$$
$$
q=4 \; : \; n=t \mbox{ or } 2t \mbox{ for } t \in \{ 1, 7, 15,  23,  31, 111 \}, \mbox{ or } t= 4^b+1 \mbox{ or } 2 \cdot 4^b+1, 
$$
$$
q=5 \; : \; n \in \{ 1, 4, 8, 9, 12, 16, 23, 24, 56, 177 \} \mbox{ or } n= 5^b+1, 2 \cdot 5^b+1 \mbox{ or }  5^b + 2, 
$$
$$
\begin{array}{c}
q=8 \; : \; n \leq 63, \; n \in \{ 92, 111, 124, 126,  158, 188, 316, 444, 479, 508, 10837 \} \\
\mbox{ or } n = r \cdot 8^b+s \mbox{ for } r, s \in \{ 1, 2, 4 \} 
\end{array}
$$
and
$$
\begin{array}{c}
q=16 \; : \; n=t, 2t \mbox{ or } 4t \mbox{ for } t  \leq 100, \;  t \in \{  111, 125, 126, 127 \} \\
\mbox{ or } t= r \cdot 16^b+s  \mbox{ where either  } r, s \in \{ 1, 2, 4, 8 \} \mbox{ or the set } \\
\{ r,s \} \mbox{ is one of }  \{ 1,3 \}, \{ 2, 3 \}, \{ 3, 8 \}, \{ 2, 12 \}, \{ 4, 12 \} \mbox{ or } \{ 8, 12 \}. \\
\end{array}
$$
Here, $b$ is a nonnegative integer. 
\end{theorem}

\noindent This immediately implies
\begin{corollary} \label{Cor2}
We have
$$
B_3(2) = \{ 7, 23 \}, \; B_3(3) = \{ 13 \}, \; B_3(4) = \{ 23, 30, 31, 46, 62, 111, 222 \},
$$
$$
 B_3(5) = \{ 56, 177 \}, \;  B_3(8) = \{ 92, 111, 124, 126, 158, 188, 316, 444, 479, 508, 10837 \}
$$
and
$$
\begin{array}{r}
B_3 (16) =  \left\{ 364, 444, 446, 500, 504, 508, 574, 628, 680, 760, 812, 888, 924, 958, \right. \\
 \left.  1012, 1016, 1020, 1022, 1784, 2296, 3832, 3966, 4088, 10837, 15864, 43348 \right\}.
\end{array}
$$
\end{corollary}

We note that the case $q=2$ of Theorem \ref{Cor1}  was originally proved by Szalay \cite{Sz02} in 2002, through appeal to a result of Beukers \cite{Beu81}. This latter work was based upon Pad\'e approximation to the binomial function (as are the results of the paper at hand, though our argument is quite distinct). In 2012, the first author \cite{Be}
treated the case $q=3$ in Theorem \ref{Cor1}. We should point out that there are computational errors in the last two displayed equations on page 4 of \cite{Be} that require repair; we will do this in the current paper.

Our main result which leads to Theorem \ref{Cor1}  is actually rather more general -- we state it for a prime base, though our arguments extend to more general $q$ with the property that $q$ has a prime-power divisor $p^\alpha$ with $p^\alpha > q^{3/4}$. We prove
\begin{theorem} \label{main-theorem}
If $q$ is an odd prime, if we have a solution to the equation
\begin{equation} \label{three-digit2}
Y^2 = t^2 + M q^{m} + N q^{n},
\end{equation}
in integers $Y, t, M, N, m$ and $n$ satisfying 
\begin{equation} \label{conditions}
t, Y, N  \geq 1, \; \;  |M|, N, t^2 \leq q-1 \;  \mbox{ and } \; \; 1 \leq m < n,
\end{equation}
then either $n=2m$ and
$Y = q^m \cdot Y_0 \pm t$, for integers $t$ and $Y_0$ with $\max \{ Y_0^2, 2 t Y_0 \} < q$, or we have $m \leq 3$.
\end{theorem}

In the special case $t=1, M = \pm 1, N=1$, a sharper version of this result already appears as the main theorem of Luca \cite{Luca}; the proof of this result relies upon primitive divisors in binary recurrence sequences and does not apparently generalize.
It seems likely that the last upper bound  in Theorem \ref{main-theorem} can be replaced by $m \leq 2$; indeed our argument can be sharpened to prove this for ``many'' pairs $(m,n)$, though not all. We know of a number of families of solutions to (\ref{three-digit2}), with, for instance, $(m,n) = (2,6)$, $q=r^2+1$ prime, $r \in \mathbb{Z}$ :
\begin{equation} \label{huge}
\left( \frac{1}{2} r ( r^6+5r^4+7r^2+5) \right)^2 = r^2 + (r^2-1) q^2 + \left( \frac{r^2+4}{4} \right) q^6
\end{equation}
and $(m,n)=(1,5)$, for  $q=64 r^2+1$, corresponding to the identity
$$
\left(r (32768r^4+1280r^2+15 ) \right)^2 = 9r^2 - (40r^2+1) q + q^5.
$$
Further families with $(m,n) = (1,3), (2,3)$ and $(1,4)$ are readily observed (as are many more examples with $(m,n)=(1,5)$). Beyond these, we also know a few (possibly) sporadic examples, with $(m,n) = (1,6), (1,7)$ and $(2,7)$ :
$$
430683365^2=9^2- 51 \cdot 311 + 205 \cdot 311^6,
$$
$$
6342918641^2 = 25^2 - 97 \cdot 673 + 433 \cdot 673^6,
$$
$$
49393781643^2 = 34^2 - 875 \cdot 1229 + 708 \cdot 1229^6,
$$
$$
559^2 = 1^2 -4 \cdot 5 + 4 \cdot 5^7,
$$
$$
574588^2 = 3^2 + 13 \cdot 31 + 12 \cdot 31^7,
$$
$$
1815^2 = 2^2 + 7^2 + 4 \cdot 7^7
$$
and
$$
20958^2 = 2^2 - 11 \cdot 13^2 + 7 \cdot 13^7.
$$

For a fixed odd prime $q$, Theorem \ref{main-theorem} provides an effective way to completely solve equation (\ref{three-digit2}) under the conditions of  (\ref{conditions}). Indeed, given an upper bound upon $m$, say $m_0$, solving (\ref{three-digit2}) with (\ref{conditions}) amounts to treating at most $O(q^{5/2} m_0)$ ``Ramanujan-Nagell'' equations of the shape
\begin{equation} \label{RamNag}
Y^2 +D = N q^n \; \; \mbox{ where } \; \; D = -(t^2 + M q^m).
\end{equation}
These can be handled efficiently via algorithms from Diophantine approximation; see Peth\H{o} and de Weger \cite{PW} or de Weger \cite{Weg0} for details.
Alternatively, if $n \equiv n_0 \mod{3}$, where $n_0 \in \{ 0, 1, 2 \}$, we may rewrite (\ref{three-digit2}) as
\begin{equation} \label{uwe}
U^2 = V^3+k,
\end{equation}
where
\begin{equation} \label{ewe2}
U = N q^{n_0} Y, \; \; V = q^{\frac{n+2n_0}{3}} N \; \; \mbox{ and } \; \; k = N^2 q^{2n_0} \left( t^2 + M q^m \right).
\end{equation}
We can therefore solve the equation (\ref{three-digit2}) if we are able to find the ``integer points'' on at most $O(q^{5/2} m_0)$ ``Mordell curves'' of the shape (\ref{uwe}), where we may subsequently check to see if any solutions encountered satisfy (\ref{ewe2}).  The integer points on these curves are known for $|k| \leq 10^7$ (see \cite{BeGh}) and are listed at
\url{http://www.math.ubc. ca/~bennett/BeGa-data.html}. For larger values of $|k|$, one can, in many cases, employ Magma or a similar computational package to solve equations of the shape (\ref{uwe}). For our purposes, however, we are led to consider a number of values of $k$ for which approaches to solving (\ref{uwe}) reliant upon computation of a full Mordell-Weil basis (as Magma does) for the corresponding curve are extremely time-consuming. We instead choose to solve a number of equations of the form (\ref{RamNag}), via lower bounds for linear forms in $p$-adic logarithms and reduction techniques from Diophantine approximation, as in \cite{PW}. An alternative approach, at least for the equations we encounter, would be to appeal to strictly elementary properties of the corresponding binary recurrences, as in a paper of Bright \cite{Bri} on the Ramanujan-Nagell equation.

It is probably worth mentioning that similar problems to those discussed in this paper, only for higher powers with few digits, are treated in a series of papers by the first author, together with Yann Bugeaud \cite{BeBu} and with Bugeaud  and Maurice Mignotte \cite{BeBuMi1}, \cite{BeBuMi2}. The results therein require rather different techniques than those employed here, focussing on lower bounds for linear forms in logarithm, $p$-adic and complex.

\section{Three digits, without loss of generality}

Suppose that $q > 1$ is an integer and that we have a square $y^2$ with (at most) three nonzero base-$q$ digits. If $q$ is either squarefree or a square, it follows that $y$ is necessarily a multiple by some power of $q$ (or $\sqrt{q}$ if $q$ is a square) of an integer $Y$ satisfying a Diophantine equation of the shape
\begin{equation} \label{wilbur}
Y^2 = C + M \cdot q^m + N \cdot q^n,
\end{equation}
where $C, M, N, m$ and $n$ are nonnegative integers with 
\begin{equation} \label{pulp}
C, M, N \leq q-1 \; \; \mbox{ and } \; \; 1 \leq m < n.
\end{equation}
If $q$ is neither a square nor squarefree, we may similarly reduce to consideration of equation (\ref{wilbur}), only with weaker bounds for $M$ and $N$. 

The machinery we will employ to prove Theorems \ref{Cor1} and \ref{main-theorem} requires that, additionally, the integer $C$  in equation (\ref{wilbur}) is square. 
Whilst this is certainly without loss of generality if every quadratic residue modulo $q$  in the range $1 \leq C < q$ is itself a square, it is easy to show that such a condition is satisfied only for $q \in \{ 2, 3, 4, 5, 8, 16 \}$. If we have the somewhat weaker constraint upon $q$ that every least positive quadratic residue $C$ modulo $q$ is either a square or has the property that it fails to be a quadratic residue modulo $q^k$ for some exponent $k>1$, then we may reduce to consideration of (\ref{wilbur}) with either $C$ square, or $m$ bounded. This weaker condition is satisfied for the following $q$ : 
$$
\begin{array}{c}
q=2, 3, 4, 5, 6, 8, 10, 12, 14, 15, 16, 18, 20, 21, 22, 24, 28, 30, 36, 40, 42, 44, 48, 54, 56, \\
60, 66, 70, 72, 78, 84, 88, 90, 102, 120, 126, 140, 150, 156, 168, 174, 180, 210, 240, \\
330, 390, 420, 462, 630, 660, 840, 2310.
\end{array}
$$
Of these, the only ones with a prime power divisor $p^\alpha$ with $p^\alpha > q^{3/4}$ (another requirement for our techniques to enable the complete determination of squares with three base-$q$ digits) are
$$
q =2, 3, 4, 5, 8, 16, 18, 22  \mbox{ and }  54.
$$

The principal reason we restrict our attention to equation (\ref{wilbur}) with $C$ square is to guarantee that the exponent $n$ is relatively large compared to $m$, enabling us to employ machinery from Diophantine approximation (this is essentially the content of Section \ref{gaps}). This might not occur if $C$ is nonsquare, as  examples like
$$
45454^2=13 + 22 \cdot 23^5 + 13 \cdot 23^6
$$
and
$$
9730060^2 = 46 + 96 \cdot 131^5 + 18 \cdot 131^6
$$
illustrate.

\section{Three digits : gaps between exponents} \label{gaps}

For the next few sections, we will restrict attention to the case where the base $q$ is an odd prime. 
Let us now suppose that we have a solution to (\ref{three-digit2}) with (\ref{conditions}). In this section, we will show that necessarily the ratio $n/m$ is not too small, except when $Y = q^m \cdot Y_0 \pm t$ for small $Y_0$. Specifically, we will prove the following result. 

\begin{lemma} \label{main-gap}
If there exists a solution to equation (\ref{three-digit2}) with (\ref{conditions}) and $m \geq 4$, then either $n=2m$ and
$Y = q^m \cdot Y_0 \pm t$, for integers $t$ and $Y_0$ with $\max \{ Y_0^2, 2 t Y_0 \} < q$, or we have $n \geq 10m-10$.
\end{lemma}
Let us begin by considering the case where $M=0$ (where we will relax the condition that $n \geq 2$). Since $q$ is an odd prime, we may write
$$
Y = q^{n} \cdot Y_0 + (-1)^\delta t,
$$
for some positive integer $Y_0$ and $\delta \in \{ 0,1 \}$, whence
$$
N =  q^{n} \cdot Y_0^2 + (-1)^\delta 2 t \cdot Y_0.
$$
Since $1 \leq N, t^2 \leq q-1$, if $n \geq 2$, it follows that
$$
q-1 \geq q^2 - 2 \sqrt{q-1},
$$
a contradiction since $q \geq 3$. We thus have $n=1$, so that
$$
N = q \cdot Y_0^2 + (-1)^\delta 2 t \cdot Y_0,
$$
whence $N < q$ implies that $Y_0=\delta =1$, corresponding to the identities
$$
\left( q - t \right)^2 = t^2 + (q-2t) q.
$$
It is worth observing that whilst there are no solutions to (\ref{wilbur}) with (\ref{pulp}), $q$ an odd prime and $M=0$, provided $C$ is square, this is not true without this last restriction, as the identity
$$
32330691^2 = 182 + 157 \cdot 367^5
$$
illustrates.

We may thus, without loss of generality, suppose that $M \neq 0$ in what follows
and write
$$
Y = q^{m} \cdot Y_0 + (-1)^\delta t,
$$
for some positive integer $Y_0$ and $\delta \in \{ 0,1 \}$, so that 
\begin{equation} \label{pengy}
q^{m} Y_0^2 + 2 (-1)^\delta t  \cdot Y_0 =  M + N q^{n-m}.
\end{equation}
We thus have 
$$
q^m - 2 q^{1/2} < q^{n-m+1} - q^{n-m} + q.
$$
If $n \leq 2m-2$ (so that $m \geq 3$), it follows that $q^m - 2 q^{1/2} < q^{m-1} - q^{m-2}+q$, an immediate contradiction.
If $n = 2m-1$, then
$$
q^{m-1}  < q+ 2 q^{1/2},
$$
and so $m=2$, $n=3$, whereby (\ref{pengy}) becomes
$$
q^{2} Y_0^2 + 2 (-1)^\delta t  \cdot Y_0 =  M + N q \leq (q-1) q + q-1 = q^2-1.
$$
We thus have $Y_0=1$ and  $\delta = 1$. Since $q \mid M - 2 (-1)^\delta t =M+2t$, it follows that either $M=-2t$ or $M=q-2t$. In the first case, we have that $q \mid N$, a contradiction. The second corresponds to the identity
\begin{equation} \label{ident-7}
(q^2-t)^2 = t^2 + (q-2t) q^2 + (q-1) q^3.
\end{equation}

Otherwise, we may suppose that $n \geq 2m$.
From the series expansion
$$
(t^2+x)^{1/2} = t + \frac{x}{2t} - \frac{x^2}{8 t^3} + \frac{x^3}{16t^5} - \frac{5x^4}{128 t^7} + \frac{7 x^5}{256 t^9} - \frac{21 x^6}{1024 t^{11}} + 
\frac{33 x^7}{2048t^{13}} - \frac{429 x^8}{32768 t^{15}} +
 \cdots,
$$
 and (\ref{three-digit2}), it follows that 
$$
Y \equiv (-1)^\delta \left( t + \frac{Mq^m}{2t} \right) \mod{q^{2m}},
$$
so that
$$
2t Y \equiv (-1)^\delta \left( 2t^2 + Mq^m \right) \mod{q^{2m}}.
$$ 
If $2tY = (-1)^\delta \left( 2t^2 + Mq^m \right)$, then
$$
n=2m, \; \; \frac{M^2}{4t^2} = N  \; \; \mbox{ and } \; \; |Y_0| = \left| \frac{M}{2t} \right|,
$$
corresponding to the identity
\begin{equation} \label{identity1}
\left( q^{m} \cdot Y_0 + (-1)^\delta t \right)^2 =t^2 + (  (-1)^\delta t  2 Y_0) \cdot q^m +   Y_0^2 \cdot q^{2m},
\end{equation}
where $\max \{ t^2, Y_0^2, 2 t Y_0 \} < q$.

If we are not in situation (\ref{identity1}), we may write
\begin{equation} \label{fishie}
2tY = \kappa q^{2m} + (-1)^\delta (M q^m + 2t^2),
\end{equation}
for some positive integer $\kappa$, so that
\begin{equation} \label{bound-2b}
4 t^2 \cdot N \cdot q^{n-2m}  
= \kappa^2 q^{2m} +2 \kappa (-1)^\delta \left( M q^m+2 t^2 \right) + M^2.
\end{equation}
We rewrite this as
\begin{equation} \label{square}
4 t^2 \cdot N \cdot q^{n-2m}  
= \left( \kappa q^m + (-1)^\delta M \right)^2 + \kappa (-1)^\delta 4 t^2.
\end{equation}
If $n=2m$, this becomes
$$
4 t^2 \cdot N 
= \left( \kappa q^m + (-1)^\delta M \right)^2 + \kappa (-1)^\delta 4 t^2,
$$
the left-hand-side of which is at most $4 (q-1)^2$. Since the right-hand-side is at least
$$
(\kappa q^m-q+1)^2-4 (q-1),
$$
it follows that $m=1$ and $\kappa \in \{ 1, 2 \}$. If $\kappa =1$, we have
$$
q + (-1)^\delta M \equiv 0 \mod{2t},
$$
say $q = 2 t q_0 - (-1)^\delta M$, for $q_0$ a positive integer with $N = q_0^2 + (-1)^\delta$,
with corresponding identity
\begin{equation} \label{ident-4}
\left( q_0 q + (-1)^\delta t \right)^2 = t^2 + (-1)^\delta (2t q_0-q) q + (q_0^2 + (-1)^\delta) q^2,
\end{equation}
where $t, q_0 < \sqrt{q}$. If $\kappa=2$,  then $M$ is necessarily even, say $M=2 M_0$, and
$$
q + (-1)^\delta M_0 \equiv 0 \mod{t},
$$
say $q= t q_0 - (-1)^\delta M_0$. This corresponds to
\begin{equation} \label{ident-3}
\left( q_0 q + (-1)^\delta t \right)^2 = t^2 + (-1)^\delta 2 (t q_0-q) q + (q_0^2 + 2 (-1)^\delta) q^2,
\end{equation}
where we require that $q/2 < t q_0 < 3q/2$, $t < \sqrt{q}$ and $q_0 < \sqrt{q - 2(-1)^\delta}$.

With these families excluded, we may thus assume that $n \geq 2m+1$ and that (\ref{square}) is satisfied.
For the remainder of this section, we will suppose that $m \geq 4$.
Then, since the right-hand-side of (\ref{bound-2b})  is 
\begin{equation} \label{wacky}
\kappa^2 q^{2m} \left( 1 +(-1)^\delta \left( \frac{2 M}{\kappa} q^{-m} + \frac{4t^2}{\kappa} q^{-2m} \right) + \frac{M^2}{\kappa^2} q^{-2m} \right),
\end{equation}
and we assume that $|M| < q$ and  $t < \sqrt{q}$, we have
\begin{equation} \label{fudge}
N \cdot q^{n-2m} > \frac{1-2 q^{1-m}-4 q^{1-2m}}{4} \, q^{2m-1}.
\end{equation}
Since $N < q$, $m \geq 4$ and $q \geq 3$ this implies that
$$
q^{n-2m+1} > \frac{2021}{8748} q^{2m-1}
$$
and hence $n \geq 4m-3 \geq 3m+1$. 
We  thus have
$$
Y \equiv (-1)^\delta \left( t + \frac{Mq^m}{2t} - \frac{M^2 q^{2m}}{8t^3} \right) \mod{q^{3m}}, 
$$
whence
$$
8 t^3 Y  \equiv (-1)^\delta \left( 8t^4 + 4t^2 Mq^m - M^2 q^{2m} \right) \mod{q^{3m}}.
$$
If
$$
8 t^3 Y  = (-1)^\delta \left( 8t^4 + 4t^2 Mq^m - M^2 q^{2m} \right),
$$
then 
$$
64 t^6 \cdot N \cdot q^{n-3m} = M^4 q^{m}  - 8 t^2 \cdot M^3,
$$
an immediate contradiction, since $n \geq 3m+1$ and $q$ is coprime to $tM$.

We may thus assume that
$$
8 t^3 Y  =\kappa_1 q^{3m}  + (-1)^\delta \left( - M^2 q^{2m} + 4t^2 Mq^m + 8t^4 \right),
$$
for a positive integer $\kappa_1$,
whereby
\begin{equation} \label{hurt}
\begin{array}{c}
64 t^6 N q^{n-3m}  = \kappa_1^2 q^{3m} + M^4 q^{m} - 8t^2 M^3  \\
 + (-1)^\delta  \left(-2 \kappa_1 M^2 q^{2m} + 8 t^2 \kappa_1 M q^{m} + 16 t^4 \kappa_1  \right) \\
\end{array}
\end{equation}
and so
\begin{equation} \label{trouble}
64 t^6 N q^{n-3m} > q^{3m} - 2 M^2 q^{2m} - 8 t^2 |M| q^m.
\end{equation}
This implies that
\begin{equation} \label{frog}
64 q^{n-3m+4} > q^{3m} \left( 1 - 2 q^{2-m} - 8 q^{2-2m} \right).
\end{equation}
For $q \geq 7$, we therefore have
$$
q^{n-3m+4} > \frac{1}{67} \, q^{3m},
$$
so that  $n \geq 6m-4$ if $q \geq 67$.  If $q=3$, we obtain the inequality  $n \geq 6m-4$ directly from (\ref{trouble}). For
each $5 \leq q \leq 61$, (\ref{frog}) implies that $n \geq 6m-6$. In every case, we may thus assume that $n \geq 6m-6 > 4m$, so that
$$
Y \equiv (-1)^\delta \left( t + \frac{Mq^m}{2t} - \frac{M^2 q^{2m}}{8t^3} + \frac{M^3 q^{3m}}{16 t^5} \right) \mod{q^{4m}}
$$
and hence 
$$
16t^5 Y = \kappa_2 q^{4m} + (-1)^\delta  \left(  16t^6 + 8t^4Mq^m - 2 t^2 M^2 q^{2m} + M^3 q^{3m} \right)
$$
for a nonegative integer $\kappa_2$, whence
\begin{equation} \label{flounder}
\begin{array}{c}
256 t^{10} N q^{n-4m} =  \kappa_2^2 q^{4m} + (-1)^\delta \left( 32 \kappa_2 t^6  + 16 \kappa_2 t^4 M q^{m}  \right. \\
\left. -4 \kappa_2 t^2 M^2 q^{2m} + 2 \kappa_2 M^3 q^{3m} \right) +  20t^4M^4-4t^2M^5 q^{m}+M^6 q^{2m}.\\
\end{array}
\end{equation}

If $\kappa_2=0$, 
$$
256 t^{10} N q^{n-4m} = 20t^4M^4-4t^2M^5 q^{m}+M^6 q^{2m},
$$
contradicting the fact that $q \not \; \mid t M$.
We therefore have that
\begin{equation} \label{sole}
256 t^{10} N q^{n-4m}  > q^{4m} - 2 |M|^3 q^{3m} - 4 t^2 M^2 q^{2m}
\end{equation}
and so
\begin{equation} \label{brown}
q^{n-4m+6} > \frac{1}{263} q^{4m},
\end{equation}
whence  $n \geq 8m-8$ unless, possibly, $q \in \{ 3, 5 \}$. If $q=3$, since $t=1$ and $|M|, N \leq 2$, inequality (\ref{sole}) implies a stronger inequality. If $q=5$, $t \leq 2$, $|M|, N \leq 4$ and inequality (\ref{sole}) again yield $n \geq 8m-8$ and hence we may conclude, in all cases that, provided $m \geq 4$, we have
$n \geq 8m-8 \geq 6m$.

From (\ref{flounder}), we  have
\begin{equation} \label{peach-1}
(-1)^\delta 8 \kappa_2 t^2 +   5 M^4 \equiv 0 \mod{q^m}.
\end{equation}
If this is equality, we must have $\delta=1$ and so (\ref{flounder}) becomes
$$
\begin{array}{c}
256 t^{10} N q^{n-5m} =  \kappa_2^2 q^{3m} - 16 \kappa_2 t^4 M   
+4 \kappa_2 t^2 M^2 q^{m} - 2 \kappa_2 M^3 q^{2m}  -4t^2M^5 +M^6 q^{m}.\\
\end{array}
$$
It follows that
\begin{equation} \label{peach-2}
4 \kappa_2 t^2 + M^4 \equiv 0 \mod{q^{m}}.
\end{equation}
Combining (\ref{peach-1}) and (\ref{peach-2}), we thus have
$$
7 M^4 \equiv 0 \mod{q^{m}},
$$
contradicting the fact that $m \geq 4$, while $0 < |M| < q$. 

We thus have 
\begin{equation} \label{puppy}
(-1)^\delta 8 \kappa_2 t^2 +   5 M^4 = \upsilon q^m
\end{equation}
for some nonzero integer $\upsilon$.
If $\upsilon$ is negative, necessarily $\kappa_2 > \frac{q^m}{8t^2}$. If $\upsilon \geq 6$, we have that, again, $\kappa_2 > \frac{q^m}{8t^2}$. Let us therefore assume that $1 \leq \upsilon \leq 5$.
Now (\ref{flounder}) is
$$
\begin{array}{c}
256 t^{10} N q^{n-5m} =  \kappa_2^2 q^{3m} + 4 t^4 \upsilon +  (-1)^\delta \left(16 \kappa_2 t^4 M   \right. \\
\left. -4 \kappa_2 t^2 M^2 q^{m} + 2 \kappa_2 M^3 q^{2m} \right) -4t^2M^5+M^6 q^{m}\\
\end{array}
$$
and so, since $n \geq 6m$, 
$$
t^2 \upsilon +  (-1)^\delta 4 \kappa_2 t^2 M - M^5 \equiv 0 \mod{q^m}.
$$
From (\ref{peach-1}), we therefore have
\begin{equation} \label{uppy}
5 \upsilon +  (-1)^\delta 28 \kappa_2 M  \equiv 0 \mod{q^m}.
\end{equation}
Since $1 \leq \upsilon \leq 5$, the left hand side here is nonzero and so 
$$
28 \kappa_2 |M| \geq q^m-25.
$$
For $q^m \geq 375$, it follows immediately
that
\begin{equation} \label{kap2}
\kappa_2 > \frac{q^{m-1}}{30},
\end{equation}
whilst the inequality if trivial if $q=3$ and $m=4$. If $q=3$ and $m=5$, we check that for $|M| \in \{ 1, 2 \}$ and $1 \leq \upsilon \leq 5$, the smallest positive solution to the congruence (\ref{uppy}) has $\kappa_2 \geq 17$, whereby (\ref{kap2}) is again satisfied.

Combining this with (\ref{flounder}), we have that
\begin{equation} \label{uber}
256 t^{10} N q^{n-4m}  > \frac{1}{900} q^{6m-2} - \frac{1}{15} |M|^3 q^{4m-1} - \frac{2}{15} t^2 M^2 q^{3m-1} - \frac{8}{15} t^4 |M| q^{2m-1},
\end{equation}
whence
$$
q^{n-4m+6}  > \frac{1}{480^2} q^{6m-2}
\left( 1 - 60 q^{4-2m} - 120 q^{4-3m} - 480 q^{4-4m} \right).
$$
It follows that  
\begin{equation} \label{lower-bound}
n \geq 10m-10
\end{equation}
if $q \geq 23$.

We note that, combining (\ref{puppy}) and (\ref{uppy}), we have
\begin{equation} \label{booger}
2 \upsilon t^2 \equiv 7 M^5 \mod{ q^{m-\delta_5} },
\end{equation}
where $\delta_5 = 1$ if $q=5$ and $0$ otherwise. For $q=3$, we have $t = 1$, $M = \pm 1, \pm 2$, and find that
$\upsilon \equiv \pm 37 \mod{81}$ if $|M|=1$ and $\upsilon \equiv \pm 31 \mod{81}$ if $|M|=2$. In all cases, from (\ref{puppy}), we have
$$
\kappa_2 \geq \frac{1}{8} \left( 31 \cdot 3^m - 80 \right) > \frac{15}{4} 3^m.
$$
Together with (\ref{flounder}), we find, after a little work, that, again, $n \geq 10m-10$. If $q=5$, congruence (\ref{booger}) implies that $|\upsilon| \geq 13$, so that
(\ref{puppy}) yields, crudely,
$$
\kappa_2 \geq \frac{1}{32} \left( 13 \cdot 5^m - 1280 \right) > \frac{1}{3} 5^m,
$$
which again, with (\ref{flounder}), implies (\ref{lower-bound}). Arguing similarly for the remaining values of $q$ with  $7 \leq q \leq 19$, enables us to conclude that inequality (\ref{lower-bound}) holds for all $q \geq 3$ and $m \geq 4$. This concludes the proof of Lemma \ref{main-gap}.

\section{Pad\'e approximants to the binomial function}

We now consider Pad\'e approximants to $(1+x)^{1/2}$, defined,
for $n_1$ and $n_2$ nonnegative integers,  via
\begin{equation} \label{pee}
P_{n_1,n_2} (x) = \sum_{k=0}^{n_1} \binom{n_2 + 1/2}{k} \binom{n_1+n_2-k}{n_2} x^k
\end{equation}
and
\begin{equation} \label{queu}
Q_{n_1,n_2} (x) = \sum_{k=0}^{n_2} \binom{n_1 - 1/2}{k} \binom{n_1+n_2-k}{n_1} x^k.
\end{equation}
As in \cite{BB}, we find that
\begin{equation} \label{aye}
 P_{n_1,n_2} (x) - \left( 1+x \right)^{1/2} \; Q_{n_1,n_2} (x) = x^{n_1+n_2+1} \, E_{n_1,n_2} (x),
\end{equation}
where (see e.g. Beukers \cite{Beu81})
\begin{equation} \label{frump}
E_{n_1,n_2} (x) =   \frac{(-1)^{n_2} \, \Gamma (n_2+3/2)}{\Gamma(-n_1+1/2) \Gamma(n_1+n_2+1)}  F(n_1 + 1/2,n_1+1, n_1+n_2+2,-x),
\end{equation}
for $F$ the hypergeometric function given by
$$
F(a,b,c,-x) = 1 - \frac{a \cdot b}{1 \cdot c} x + \frac{a \cdot (a+1) \cdot b \cdot (b+1)}{1 \cdot 2 \cdot c \cdot (c+1)} x^2 - \cdots.
$$
Appealing twice to  (\ref{aye}) and  (\ref{frump}) and eliminating $(1+x)^{1/2}$, the quantity 
$$
P_{n_1+1,n_2}(x)Q_{n_1,n_2+1}(x)-P_{n_1,n_2+1}(x)Q_{n_1+1,n_2}(x)
$$
is a polynomial of degree $n_1+n_2+2$ with a zero at $x=0$ of order $n_1+n_2+2$ (and hence is a monomial). It follows that we may write
\begin{equation} \label{zero}
P_{n_1+1,n_2}(x)Q_{n_1,n_2+1}(x)-P_{n_1,n_2+1}(x)Q_{n_1+1,n_2}(x) = c x^{n_1+n_2+2}.
\end{equation}
Here, we have
$$
c = (-1)^{n_2+1} \frac{(2n_1-2n_2-1) \Gamma (n_2 + 3/2)}{2 (n_1+1)! \, (n_2+1)!  \,\Gamma (-n_1+1/2)} \neq 0.
$$
We further observe that
 $$
\binom{n+ \frac{1}{2}}{k} 4^k \in \mathbb{Z},
$$
so that, in particular, if $n_1 \geq n_2$, $4^{n_1}  P_{n_1,n_2} (x)$ and $4^{n_1}  Q_{n_1,n_2} (x)$ are polynomials with integer coefficients.

\subsection{Choosing $n_1$ and $n_2$}

For our purposes, optimal choices for $n_1$ and $n_2$ are as follows (we denote by $[x]$ the greatest integer not exceeding a real number $x$ and set $x=[x]+\{x\}$).
\begin{dfn} \label{yucca}
Define
$$
(n_1,n_2) = \left( \left[ \frac{3n}{4m} \right] + \delta - \Delta_1, \left[ \frac{n}{4m} \right] - \delta + \Delta_2 \right)
$$
where $\delta \in \{ 0, 1 \}$,
$$
\Delta_1 =
\left\{
\begin{array}{cl} 
1 & \mbox{ if } \left\{ \frac{n}{4m} \right\}  \in [0,1/4] \cup [1/3,1/2] \cup [2/3,3/4] \\
0 & \mbox{ if } \left\{ \frac{n}{4m} \right\}  \in (1/4,1/3) \cup (1/2,2/3) \cup (3/4,1), \\
\end{array}
\right.
$$
and
$$
\Delta_2 =
\left\{
\begin{array}{cl} 
1 & \mbox{ if } \left\{ \frac{n}{4m} \right\} >0 \\
0 & \mbox{ if } \left\{ \frac{n}{4m} \right\} =0. \\
\end{array}
\right.
$$
\end{dfn}

Note that for these choices of $n_1$ and $n_2$, we may check that
$$
(n_1+n_2+1) m = n + \left( \Delta_2- \Delta_1 + 1 -  \left\{ \frac{n}{4m} \right\} - \left\{ \frac{3n}{4m} \right\} \right) m  \geq n.
$$
Further, we have
$$
n_1 (m+1) = \frac{3n}{4} + \frac{3n}{4m} + \kappa_1 (m,n,\delta)
$$
and
$$
n_2 (m+1) + n_1-n_2 + \frac{n}{2} = \frac{3n}{4}  + \frac{3n}{4m} + \kappa_2 (m,n,\delta),
$$
where
$$
\kappa_1 (m,n,\delta) = 
(m+1) \left( \left[ 3 \left\{ \frac{n}{4m} \right\} \right] + \delta - \Delta_1- 3  \left\{ \frac{n}{4m} \right\}  \right)
$$
and
$$
\kappa_2 (m,n,\delta) = - (m+3) \left\{ \frac{n}{4m} \right\} + \left[ 3 \left\{ \frac{n}{4m} \right\} \right] + ( \Delta_2-\delta) m + \delta - \Delta_1.
$$
A short calculation ensures that, in every situation, we have
\begin{equation} \label{fabulous}
 \max \{ n_1 (m+1),  n_2 (m+1) + n_1-n_2 + n/2 \}
 \leq \frac{3n}{4} + \frac{3n}{4m} + m- \frac{5}{4},
\end{equation}
where the right-hand-side is within $O(1/m)$ of the ``truth'' for $\delta=0$, $\Delta_1=\Delta_2=1$.

Note that the fact that $n \geq 10m-10$ implies that we have $n_2 \geq 2$, unless
$$
(m,n) \in \{ (4,30), (4,31), (4,32), (5,40) \},
$$
where we might possibly have $n_2=1$. In all cases, we also have
\begin{equation} \label{frozen}
\left| n_1 - 3 n_2 \right| \leq 3.
\end{equation}

\subsection{Bounds for $|P_{n_1,n_2} (x) |$ and $|Q_{n_1,n_2} (x) |$.}

We will have need of the following result.

\begin{lemma} \label{lemma1}
If $n_1$ and $n_2$ are as given in Definition \ref{yucca}, where $m \geq 4$ and $n \geq 10m-10$ are integers, then we have
$$
|P_{n_1,n_2} (x)| \leq  2 \, |x|^{n_1}
\; \; \mbox{ and } \; \; 
|Q_{n_1,n_2} (x)| \leq 2^{n_1+n_2-1} \left( 1 + \frac{|x|}{2} \right)^{n_2},
$$
for all real numbers $x$ with $|x| \geq 16$.
\end{lemma}

\begin{proof}
Arguing as in the proof of Lemma 1 of Beukers \cite{Beu}, we have that
$$
|Q_{n_1,n_2} (x)| \leq  \sum_{k=0}^{n_2} \binom{n_1}{k} \binom{n_1+n_2-k}{n_1} |x|^k =  \sum_{k=0}^{n_2} \binom{n_2}{k} \binom{n_1+n_2-k}{n_2} |x|^k.
$$
Since $n_1 > n_2$ and  $\binom{n_1+n_2-k}{n_2} \leq 2^{n_1+n_2-k-1}$, it follows that
$$
|Q_{n_1,n_2} (x)| \leq 2^{n_1+n_2-1} \left( 1 + \frac{|x|}{2} \right)^{n_2}.
$$

Next, note that, since $n_1 > n_2$, $|P_{n_1,n_2} (x)|$ is bounded above by 
$$
\sum_{k=0}^{n_2+1} \binom{n_2+1}{k} \binom{n_1+n_2-k}{n_2} |x|^k + \sum_{k=n_2+2}^{n_1} \frac{(n_2+1)! (k-n_2-1)!}{k!} \binom{n_1+n_2-k}{n_2} |x|^k.
$$
The first sum here is, arguing as previously,  at most
$$
2^{n_1+n_2-1} \left( 1 + \frac{|x|}{2} \right)^{n_2+1}.
$$
For the second, we split the summation into the ranges $n_2+2 \leq k \leq \left[ \frac{n_1+n_2}{2} \right]$ and $\left[ \frac{n_1+n_2}{2} \right]+1 \leq k \leq n_1$. In the second of these, we have $n_1+n_2-k < k$ and so
$$
\binom{n_1+n_2-k}{n_2} < \binom{k}{n_2},
$$
whence
$$
\sum_{k=\left[ \frac{n_1+n_2}{2} \right]+1}^{n_1} \frac{(n_2+1)! (k-n_2-1)!}{k!} \binom{n_1+n_2-k}{n_2} |x|^k
< \sum_{k=\left[ \frac{n_1+n_2}{2} \right]+1}^{n_1} \frac{n_2+1}{k-n_2} \;  |x|^k.
$$
Appealing to Definition \ref{yucca}, we may show that $2n_2 \leq \left[ \frac{n_1+n_2}{2} \right]+2$
and hence $\frac{n_2+1}{k-n_2} \leq 1$, so that
$$
\sum_{k=\left[ \frac{n_1+n_2}{2} \right]+1}^{n_1} \frac{n_2+1}{k-n_2} \;  |x|^k \leq
\sum_{k=\left[ \frac{n_1+n_2}{2} \right]+1}^{n_1}   |x|^k <
\frac{|x|}{|x|-1} \; |x|^{n_1},
$$
provided $|x| > 1$.
Since
$$
\sum_{k=n_2+2}^{\left[ \frac{n_1+n_2}{2} \right]} \frac{(n_2+1)! (k-n_2-1)!}{k!} \binom{n_1+n_2-k}{n_2} |x|^k
<
\sum_{k=n_2+2}^{\left[ \frac{n_1+n_2}{2} \right]} \binom{n_1+n_2-k}{n_2} |x|^k
$$
and
$$
\sum_{k=n_2+2}^{\left[ \frac{n_1+n_2}{2} \right]} \binom{n_1+n_2-k}{n_2} |x|^k
\leq 
\sum_{k=n_2+2}^{\left[ \frac{n_1+n_2}{2} \right]}2^{n_1+n_2-k-1} \, |x|^k
< \sum_{k=n_2+2}^{\left[ \frac{n_1+n_2}{2} \right]}  |2x|^k,
$$
we may conclude that $|P_{n_1,n_2} (x)|$ is bounded above by
$$
2^{n_1+n_2-1} \left( 1 + \frac{|x|}{2} \right)^{n_2+1} + 
\frac{|x|}{|x|-1} \; |x|^{n_1} + \frac{|2x|}{|2x|-1} \; |2x|^{ \frac{n_1+n_2}{2}}.
$$
Since $|x| \geq 16$ and, via (\ref{frozen}), $n_1 \geq 3n_2-3$, checking values with $n_2 \leq 10$ separately, we may conclude that
$$
|P_{n_1,n_2} (x)| < 2 \, |x|^{n_1}.
$$
This concludes our proof.
\end{proof}

\section{Proof of Theorem \ref{main-theorem}}  \label{guppy}

To prove Theorem \ref{main-theorem}, we will, through the explicit Pad\'e approximants of the preceding section, construct an integer that is nonzero and, in archimedean absolute value ``not too big'', while, under the assumptions of the theorem, being divisible by a very large power of our prime $q$. With care, this will lead to the desired contradiction.

Setting
$\eta = \sqrt{t^2+M q^m}$,
since $(1+x)^{1/2}$, $P_{n_1,n_2} (x)$ and $Q_{n_1,n_2} (x)$ have $q$-adic integral coefficients, the same is also true of $E_{n_1,n_2} (x)$ and so, via equation (\ref{aye}),
$$
\left|  t P_{n_1,n_2} \left( \frac{M q^m}{t^2} \right) - \eta \, Q_{n_1,n_2} \left( \frac{M q^m}{t^2} \right) \right|_q \leq q^{-n}.
$$
On the other hand, from the fact that 
$\eta^2 \equiv Y^2 \mod{q^n}$,
we have
$$
\eta \equiv (-1)^{\delta_1} Y \mod{q^n},
$$
for some $\delta_1 \in \{ 0, 1 \}$, and hence
$$
\left| t P_{n_1,n_2} \left( \frac{M q^m}{t^2} \right) - (-1)^{\delta_1} Y  \,  Q_{n_1,n_2} \left( \frac{M q^m}{t^2} \right) \right|_q \leq q^{-n}.
$$
Equation (\ref{zero}) implies that for at least one of our two pairs $(n_1,n_2)$, we must have
$$
t P_{n_1,n_2}  \left( \frac{M q^m}{t^2} \right)  \neq (-1)^{\delta_1} Y  \,  Q_{n_1,n_2}  \left( \frac{M q^m}{t^2} \right) 
$$
and hence, for the corresponding pair $(n_1,n_2)$, we have that 
$$
 (2t)^{2n_1}  \, P_{n_1,n_2}  \left( \frac{M q^m}{t^2} \right)  - (-1)^{\delta_1} Y  \, 2^{2n_1} \, t^{2n_1-1} \,  Q_{n_1,n_2}  \left( \frac{M q^m}{t^2} \right)
 $$
 is a nonzero integer, divisible by $q^n$, and so, in particular, 
\begin{equation} \label{lower}
\left|  (2t)^{2n_1}  \, P_{n_1,n_2}  \left( \frac{M q^m}{t^2} \right)  - (-1)^{\delta_1} Y  \, 2^{2n_1} \, t^{2n_1-1} \,  Q_{n_1,n_2}  \left( \frac{M q^m}{t^2} \right)  \right| \geq q^{n}.
\end{equation}
From Lemma \ref{lemma1} and  the fact that $Y < q^{(n+1)/2}$, we thus have
\begin{equation} \label{moose}
q^n \leq 2^{2n_1+1}  |M|^{n_1} q^{m n_1} + 2^{3n_1+n_2-1} q^{(n+1)/2} t^{2n_1-1} \left( 1 + \frac{|M|q^m}{2t^2} \right)^{n_2}.
\end{equation}
From the  inequalities
$$
|M|, t^2 < q \; \mbox{ and } \; \frac{|M|q^m}{2t^2} \geq \frac{81}{2},
$$
it follows from (\ref{moose})  that
\begin{equation} \label{real}
q^n \leq 2^{2n_1+1}  \cdot q^{(m+1) n_1} + 2^{3n_1-1} q^{n/2 + (m+1) n_2 + n_1-n_2} \, (83/81)^{n_2},
\end{equation}
and hence, since $n \geq 10m-10$ and $m \geq 4$, we may argue rather crudely to conclude that
\begin{equation} \label{bosnia}
q^n < 9^{n_1}  \cdot q^{\max \{ n_1 (m+1),  n_2 (m+1) + n_1-n_2 + n/2 \}} .
\end{equation}
Inequality (\ref{fabulous}) thus implies 
$$
q^n < 9^{\frac{3n}{4m} +1} \cdot q^{\frac{3n}{4} + \frac{3n}{4m} + m- \frac{5}{4}},
$$
whence
\begin{equation} \label{loop}
q^{1 - \frac{3}{m} - \frac{4m}{n} + \frac{5}{n}} < 9^{\frac{3}{m} + \frac{4}{n}}.
\end{equation}
Since $m \geq 4$, if $n$ is suitably large, this provides an upper bound upon $q$.
In particular, if 
\begin{equation} \label{nimh}
n > \frac{4m^2-5m}{m-3},
\end{equation}
then
\begin{equation} \label{rats}
q < 3^{\frac{6n+8m}{mn-3n-4m^2+5m}}.
\end{equation}
Since $m \geq 4$ and $n \geq 10m-10$, (\ref{nimh}) is satisfied unless we have $m=4$ and $30 \leq n \leq 44$.
Excluding these values for the moment, we thus have
$$
q < 3^{\frac{68m-60}{6m^2-35m+30}}.
$$
Since $q \geq 3$, it follows, therefore, that, in all cases, $m \leq 16$. If $q \geq 5$, we have the sharper inequality $m \leq 12$.

\subsection{Small values of $m$}

To treat the remaining values of $m$, we argue somewhat more carefully. For fixed $q$ and $m$, equation (\ref{three-digit2}) under the conditions in  (\ref{conditions}) can, in many cases, be shown to have no solutions via simple local arguments. In certain cases, however, when the tuple $(t,M,N,m)$ matches up with an actual solution, we will not be able to find such local obstructions. For example, the identities 
$$
\left(   q^m \cdot Y_0 \pm t \right)^2 = t^2  \pm 2 t Y_0 q^m + Y_0^2 q^{2m}
$$
imply that we cannot hope, through simple congruential arguments, to eliminate the cases (here $n \equiv n_0 \mod{3}$)
\begin{equation} \label{case1}
(t,M,N,n_0) = (t,\pm 2 t Y_0, Y_0^2, 2m \mod{3}),
\end{equation}
where $\max \{ t^2, Y_0^2, 2 t Y_0 \} < q$. For even values of $m$, we are also unable to summarily dismiss tuples like
\begin{equation} \label{case2}
(t,M,N,n_0) = (t,Y_0^2, 2 t Y_0, m/2 \mod{3}).
\end{equation}
Additionally, the ``trivial'' identity
$$
t^2 = t^2 - M \cdot q^m + M \cdot q^m
$$
leaves us with the necessity of treating tuples
\begin{equation} \label{case3}
(t,M,N,n_0) = (t,-N,N, m \mod{3})
\end{equation}
via other arguments. By way of example, if $q=m=5$, sieving by primes $p$ with the property that the smallest positive $t$ with $5^t \equiv 1 \mod{p}$ divides 
$300$, we find that all tuples $(t,M,N,n_0)$ are eliminated except for
$$
\begin{array} {c}
(1,-2,1,1), (1,-1,1,2), (1,2,1,1), (1,-2,2,2), (1,1,2,1), (1,-3,3,2), \\
(1,-4,4,1), (1,-4,4,2), (1,4,4,1), (2,-4,1,1), (2,-1,1,2), (2,4,1,1), \\
(2,-2,2,2), (2,-3,3,2) \; \mbox{ and } \;  (2,-4,4,2).\\
 \end{array}
$$
These all correspond to (\ref{case1}) or (\ref{case3}), except for $(t,M,N,n_0)=(1,1,2,1)$ which arises from the identity $56^2 = 1^2 + 2 \cdot 5 + 5^5$.

For the cases where we fail to obtain a local obstruction, we can instead consider  equations (\ref{uwe}), with the conditions  (\ref{ewe2}). Our expectation is that, instead of needing to treat roughly $6 (q-1)^{5/2}$ such equations (for a fixed pair $(q,m)$), after local sieving we will be left with on the order of $O(q)$ Mordell curves to handle.

By way of example, let us begin with the case where $q=3$.
Here, from (\ref{moose}),
$$
3^n \leq 2^{3n_1+1} 3^{m n_1} + 2^{3n_1+n_2-1}  (82/81)^{n_2} 3^{mn_2 + (n+1)/2}.
$$
Since $\max \{ m n_1, m n_2 + (n+1)/2 \} \leq \frac{3n}{4}+m + \frac{1}{4}$, and $n_2 \geq 2$ (provided $n > 40$), we thus have
$$
3^n \leq 2^{3n_1+n_2}  (82/81)^{n_2} 3^{\frac{3n}{4}+m + \frac{1}{4}},
$$
so that
$$
3^{n/4-m-1/4} \leq 2^{3n_1+n_2}  (82/81)^{n_2}.
$$
We check that $n_2 \leq \frac{n}{4m}+1$ and $3n_1+n_2 \leq \frac{5n}{2m} + \frac{3}{2}$, whence
either $n \leq 40$, or we have
$$
3^{\frac{n}{4}-m-\frac{1}{4}} \leq 2^{\frac{5n}{2m} + \frac{3}{2}}  (82/81)^{\frac{n}{4m}+1}.
$$
In this latter case, if $m \geq 12$, the fact that $n \geq 10m-10$ leads to a contradiction, whilst, for $8 \leq m \leq 11$, we have that $n \leq 157$. A short calculation ensures that there are no solutions to equation (\ref{three-digit2}) with (\ref{conditions}), if $q=3$, $8 \leq m \leq 11$ and $10m-10 \leq n \leq 157$.
For $q=3$ and $4 \leq m \leq 7$, we are led to equation of the shape (\ref{uwe}), where now  $|k| \leq 324 \left( 1 + 2  \cdot 3^m \right) \leq 1417500$. As noted previously, the integer points on the corresponding Mordell curves are known (see \cite{BeGh}) and listed at
\url{http://www.math.ubc.ca/~bennett/BeGa-data.html}.
We check that no solutions exist with $U$ and $V$ as in (\ref{ewe2}).

We may thus suppose that $q \geq 5$ and hence it remains to treat the values of $m$ with $4 \leq m \leq 12$.
If $m=12$, appealing to (\ref{rats}), we have, from the fact that $n \geq 110$, necessarily $110 \leq n \leq 118$ and $q=5$. A short calculation ensures that there are no corresponding solutions to equation (\ref{three-digit2}) with (\ref{conditions}). Similarly, if $m=11$, we have that either $q=5$ and $100 \leq n \leq 125$, or $q=7$, $100 \leq n \leq 103$. If $m=10$, $q=5$ and $90 \leq n \leq 139$, or $q=7$ and $90 \leq n \leq 109$, or $q=11$ and $n=90$. For $m=9$ we have, in all cases, $n \leq 172$ and $q \leq 19$. For $m=8$, $n \leq 287$ and $q \leq 47$. A modest computation confirms that we have no new solutions to the equation of interest and hence we may suppose that $4 \leq m \leq 7$ (and that $q \geq 5$). 

For small values of $q$, each choice of $m$ leads to at most $2 q^{5/2}$ Ramanujan-Nagell equations (\ref{RamNag}) which we can solve as in \cite{PW}. In practice, the great majority of these are eliminated by local sieving. By way of example, if $q=5$, after local sieving, we are left to treat precisely $32$ pairs $(D,N)$  in equation (\ref{RamNag}), corresponding to
$$
\begin{array}{c}
D \in  \left\{ -312498, -15624, -15623, -12498, -2498, -1249,  \right.\\
\hskip8ex \left. -624, 1251, 2502,  6251, 12502, 31251, 312502 \right\}, \; \mbox{ if } \; N=1, \\
\end{array}
$$
$$
D \in \{ -156248, -31248, 3126, 15626 \}, \; \mbox{ if } \; N=2, 
$$
$$
D \in \{ -234374, -234373,  -46874, -46873, -1873, 31252 \}, \; \mbox{ if } \; N=3
$$
and
$$
D \in \{ -312499, -62498, -2499, -2498, 627, 2501, 12501, 15627, 62501\} \; \mbox{ if } \; N=4.
$$
For these values of $(D,N)$, we find that equation (\ref{RamNag}) has precisely solutions as follows
$$
\begin{array}{ccc|ccc} \hline
D & N & n & D & N & n \\ \hline
-312499 & 4 &  7 & 2501 & 4 & 2 \\
-312499 & 4 & 14 & 2501 & 4 & 8 \\
-234374 &  3 & 7 & 3126 & 2 & 1 \\
-46874 & 3 & 6 & 6251 & 1 & 10 \\
-15624 & 1 & 6 & 12501 & 4 & 10 \\
-2499 & 4 & 4 & 15626 & 2 & 3 \\
-2499 & 4 & 8 & 31251 & 1 & 12 \\
-1249 & 1 & 8 & 62501 & 4 & 3 \\
-624 & 1 & 4 & 62501 & 4 & 12 \\
1251 & 1 & 8 & & & \\ \hline
\end{array}
$$
\vskip2ex
\noindent In all cases, these solutions correspond to values of $m$ that have either $m \geq n$ or $n=2m$. More generally, implementing a ``Ramanujan-Nagell'' solver as in \cite{PW}, in conjunction with local sieving, we completely solve equation  (\ref{three-digit2}) with (\ref{conditions}), for $m \in \{ 4, 5, 6, 7 \}$ and $5 \leq q \leq 31$.
No new solutions accrue. If we appeal again to inequality (\ref{rats}), using that $q \geq 37$, we find that $60 \leq n \leq 81$ (if $m=7$), $50 \leq n \leq 109$ (if $m=6$) and $40 \leq n \leq 499$ (if $m=5$). After a short computation, we are left to consider the cases with $m=4$ and $q \geq 37$.

For the value $m=4$, proceeding in this manner would entail an extremely  large computation, without additional ingredients. 
By way of example, in case $m=4$ and $n=45$, inequality (\ref{rats}) implies an upper bound upon $q$ that exceeds $10^{144}$ (and no upper bound whatsoever for $30 \leq n \leq 44$). To sharpen this and related inequalities, we will argue as follows. Notice that if we have 
\begin{equation} \label{marvel}
t P_{n_1,n_2}  \left( \frac{M q^m}{t^2} \right) = (-1)^{\delta_1} Y  \,  Q_{n_1,n_2}  \left( \frac{M q^m}{t^2} \right),
\end{equation}
then
$$
t^2 P_{n_1,n_2}^2\left( \frac{M q^m}{t^2} \right) - (t^2 + M q^m + N q^n) Q_{n_1,n_2}^2   \left( \frac{M q^m}{t^2} \right) =0.
$$
From our construction, it follows that 
$$
 \left| t^2 P_{n_1,n_2}^2\left( \frac{M q^m}{t^2} \right) - (t^2 + M q^m) Q_{n_1,n_2}^2   \left( \frac{M q^m}{t^2} \right) \right|_q \leq q^{-m (n_1+n_2+1)}.
$$
and hence, if $(n_1+n_2+1)m > n$ and (\ref{marvel}), then 
\begin{equation} \label{goof}
q^{(n_1+n_2+1)m - n} \; \mbox{ divides } \; Q_{n_1,n_2}^2 (0) = \binom{n_1+n_2}{n_2}^2.
\end{equation}
In particular, if $m=4$ and  $30 \leq n \leq 32$, then we have $(n_1,n_2) \in \{ (5,2), (6,1) \}$ and hence, since $q \geq 37$, (\ref{goof}) fails to hold. We thus obtain inequality (\ref{lower})  for both pairs $(n_1,n_2)$, rather than just for one of them, provided $n \in \{ 30, 31 \}$ (if $n=32$, we have $(n_1+n_2+1)m=n$). Choosing $(n_1,n_2)=(5,2)$, it follows from (\ref{bosnia}) that, if $n=30$, we have $q^{2} < 3^{10}$, so that $q \leq 241$, while $n=31$ implies $q^{5/2} < 3^{10}$, i.e. $q \leq 79$. If $n=32$, the worse case corresponds to $(n_1,n_2)=(6,1)$, where we find, again from (\ref{bosnia}), that $q^2 < 3^{12}$ and so $q \leq 727$. 
Continuing in this fashion, observing that the greatest prime factor $\binom{n_1+n_2}{n_2}$ is bounded above by roughly $n/4$, and that $4(n_1+n_2+1)=n$ precisely when $4 \mid n$, we have, via (\ref{bosnia}), an upper bound upon $q$ of the shape $q < \min_{\delta \in \{ 0, 1 \}} \{ 3^{2n_1/(n-\mu)} \}$, if $4 \not\;\mid n$, and $q < \max_{\delta \in \{ 0, 1 \}} \{ 3^{2n_1/(n-\mu)} \}$, if $4 \mid n$, where 
$$
\mu = \max \{ n_1 (m+1),  n_2 (m+1) + n_1-n_2 + n/2 \} .
$$
Here, we exclude the cases where $\mu \geq n$, corresponding to $(n_1,n_2)=(5,3)$ if $n=33$ or $34$ and $(n_1,n_2)=(9,2)$ if $n=45$; in each of these, the other choice of $(n_1,n_2)$ leads to a bound upon $q$. For $n \leq 1000$, we find that $q < 3^{10}$, in case $n=36$, $q < 3^{28/3}$ (if $q=41$), $q < 3^8$ (if $n=52$ or $n=57$) and otherwise $q < 3155$. A painful but straightforward computation finds that we have no additional solutions to equation  (\ref{three-digit2}) with (\ref{conditions}) for $n \leq 1000$. Applying once again inequality (\ref{rats}), we may thus assume that $q \leq 1021$. After local sieving and solving corresponding equations of the shape (\ref{RamNag}),  we verify that  equation  (\ref{three-digit2}) has no unexpected solutions with (\ref{conditions}), for $m =4$ and $37 \leq q \leq 1021$. This completes the proof of Theorem \ref{main-theorem}.

Full details of our computations are available from the authors upon request.

\section{Proof of Theorem \ref{Cor1}}

For $q \in \{ 3, 5 \}$, we may apply Theorem \ref{main-theorem} to conclude that either $n=3^b+1$ (in case $q=3$) or that $n \in \{ 5^b+1, 2 \cdot 5^b+1, 5^b+2 \}$ (if $q=5$),  for some positive integer $b$, or that we have either
\begin{equation} \label{last}
n^2 = 1 + M \cdot 3^m + N \cdot 3^n, \; n^2 = 1 + M \cdot 5^m + N \cdot 5^n \; \mbox{ or } \; n^2 = 4 + M \cdot 5^m + N \cdot 5^n,
\end{equation}
with $m \in \{ 1, 2, 3 \}$, $n > m$ and $1 \leq M, N \leq q-1$. Checking the corresponding solutions to (\ref{uwe}) (all available at \url{http://www.math.ubc.ca/~bennett/BeGa-data.html}), we find that the only solutions to (\ref{last}) are with 
$$
n \in \{ 4, 5, 8, 9, 12, 13, 16,  23, 24, 56, 177 \}, 
$$
as claimed. Adding in the ``trivial'' solutions with $n \in \{ 1, 2 \}$, completes the proof of Theorem \ref{Cor1} in case $q \in \{ 3, 5 \}$.

Our argument for $q \in \{ 2, 4, 8, 16 \}$ follows along very similar lines to the proof of Theorem \ref{main-theorem}, only with slight additional complications, arising from the fact that none of  $(1+x)^{1/2}$, $P_{n_1,n_2} (x)$ or $Q_{n_1,n_2} (x)$ have $2$-adic integral coefficients. On the other hand,  $(1+4x)^{1/2}$, $P_{n_1,n_2} (4x)$ and $Q_{n_1,n_2} (4x)$ do have $2$-adic integral coefficients and so we can proceed as in Section \ref{guppy}, taking $x = M q^m/t^2$, where now $q=2^\alpha$ for $\alpha \in \{ 1, 2, 3, 4 \}$. Under mild assumptions upon $m$ ($m \geq 5$ is satisfactory), the arguments of Sections \ref{gaps} and \ref{guppy} go through with essentially no changes. We are left to treat a number of equations of the shape (\ref{RamNag}), to complete the proof of Theorem \ref{Cor1}. We suppress the details.

\section{Concluding remarks}

In this paper, we have focussed our attention on equation (\ref{wilbur}) in case $C$ is square and $q$ is prime. Even in this very restricted situation, we have been able to use our results to completely determine $B_3(q)$ only for $q \in \{ 2, 3, 5 \}$. We conclude with some speculations upon the structure of the sets $B_3(q)$.
Let us write 
$$
B_k(q) = \bigcup_{j=k}^\infty B_{k,j}(q),
$$
where
$$
B_{k,j} (q) = \left\{ n \in \mathbb{N} \; : \; n \not\equiv 0 \mod{q}, \; N_q(n)=j \; \mbox{ and } \; N_q(n^2)=k \right\}.
$$
If $q=r^2+1$ is prime for $r$ an integer, since we have
$$
\frac{1}{2} r ( r^6+5r^4+7r^2+5) = r + r \cdot q^2 + \frac{r}{2} \cdot q^3,
$$
 identity (\ref{huge}) implies that $B_{3,3} (q)$ is nonempty  for such $q$. Further, for odd prime $q$, we can find examples to verify that
 $B_{3,4}(q)$ is  nonempty for (at least)
 $$
 q = 7, 11, 17, 23, 31, 47, 101, 131, 151,
 $$
amongst the primes up to $200$.
We observe that
$$
35864 \in B_{3,5}(11).
$$
We know of no other value in $B_{3,j}(q)$ for $j \geq 5$ and $q$ prime. Perhaps there are none.

\section{Acknowledgments}

The authors are grateful to the referees for pointing out a number of errors, typographical and otherwise.


\end{document}